\documentclass{amsart} 
\usepackage{a4wide} 

\usepackage{amsthm}
\usepackage{amssymb}
\usepackage{amsmath} 
\usepackage{amscd}
\usepackage{color}
\usepackage{tikz-cd}
\usetikzlibrary{decorations.pathmorphing}

\newtheorem{theorem}{Theorem}
\newtheorem{lemma}[theorem]{Lemma}
\newtheorem{corollary}[theorem]{Corollary}
\newtheorem{proposition}[theorem]{Proposition}
\newtheorem{example}[theorem]{Example}
\theoremstyle{remark}

\theoremstyle{definition}

\newcommand{\Hom}{\mbox{Hom}}
\newcommand{\Mor}{\mbox{Mor}}

\newcommand{\Set}{\mbox{Set}}

\newcommand{\Z}{\mathbb{Z}}
\newcommand{\Q}{\mathbb{Q}}

\newcommand{\termin}{\text{up to zeros}}

\DeclareMathOperator{\mF}{\mathbf{F}}

\DeclareMathOperator{\id}{\rm id}

\DeclareMathOperator{\A}{\rm Act}
\DeclareMathOperator{\oA}{\overline{\rm Act}}

\begin{document}
\title{Perfect monoids with zero and categories of $S$-acts}

\author{Josef Dvo\v r\'ak}
\email{pepa.dvorak@post.cz}
\address{CTU in Prague, FEE, Department of mathematics, Technick\'a 2, 166 27 Prague 6 \&
	MFF UK, Department of Algebra,  Sokolovsk\' a 83, 186 75 Praha 8, Czechia}

\author{Jan \v Zemli\v cka}
\email{zemlicka@karlin.mff.cuni.cz}
\address{Department of Algebra, Charles University in Prague,
	Faculty of Mathematics and Physics Sokolovsk\' a 83, 186 75 Praha 8, Czechia}

\begin{abstract} 
In this paper, we study the relationship between the two main categories of $S$-acts for a monoid $S$ with zero from the viewpoint of existence of projective covers and the equivalence is proven. Furthermore, monoids with zeros over which all compact acts are cyclic are characterized.
	
\end{abstract}



\thanks{This work is part of the project SVV-2020-260589}

\maketitle

\section{Introduction}

The usefulness of the notion of projective cover within the context of module theory has been confirmed in countless occasions since the publication of the founding works of Bass \cite{Bass60}, who coined the term, and Eilenberg \cite{Eil56}, who effectively considered the notion for the first time. Together with the idea of projective cover, the closely related notion of a perfect ring, for which projective covers exist in the corresponding module category, appears. The definition of both terms uses a purely categorial language, yet structural and homological characterizations can be given, e.g.,
\begin{theorem}\label{perf_ring} \cite{Bass60}
	The following conditions are equivalent for a ring $R$
	\begin{enumerate}
		\item $R$ is left perfect
		\item $R$ satisfies d.c.c. on principal right ideals 
		\item the class of projective $R$-modules coincides with the class of flat $R$-modules.
	\end{enumerate}  
\end{theorem}

The corresponding notion within the branch of monoids and acts turned out to be similarly fruitful with applications to category theory and topological monoids (see \cite{I}). Note that the Theorem~\ref{perf_ring} has its counterpart stated for monoids:
\begin{theorem}\cite{Foun, I, K}
	The following conditions are equivalent for a monoid $S$
	\begin{enumerate}
		\item $S$ is left perfect
		\item $R$ satisfies the minimum condition on principal right ideals and each left $S$-act satisfies the a.c.c for cyclic subacts
		\item the class of projective $S$-acts coincides with the class of strongly flat left $S$-acts.
	\end{enumerate}  
\end{theorem}

The previous result as well as other results have been formulated and considered within the context of the category $S-\oA$ (see below), but for a monoid with zero, the monograph \cite{KKM} introduces another natural category, $S-\A_0$, which turns out to possess notably different categorial properties regarding e.g. its extensivity or compactness of objects (cf. \cite{DZ21}), hence the question of relationship of these two categories from the viewpoint of perfectness arises naturally and the aim of the present paper is an investigation on this topic.

The question of perfectness appears to be related to the problem over which monoids $S$ (or rings) there exists a non-cyclic act such that the corresponding covariant 
$\Hom$-functor from a category of $S$-acts (or $S$-modules) commutes with coproducts (such monoid is then called non-steady). 
It is known that non-steady  monoids are necessarily non-perfect 
in the category $S-\oA$ (as well as in the case of modules).

The main tool of the paper is the functor $\mF:  S-\oA\to S-\A_0$
gluing all zero elements to one using Rees factor. It allows translating the properties of $S-\oA$ to the category $S-\A_0$.
Namely, Theorem~\ref{char_0perfect} shows that the left perfectness
of categories $S-\oA$ and  $S-\A_0$ coincide and 
a monoid with zero is left perfect if and only if  it is left 0-perfect 
and it is left 0-steady if and only if it satisfies the ascending chain condition on cyclic subacts by Theorem~\ref{char_0steady}.

\section{Preliminaries}

Before we begin the exposition, let us recall some necessary terminology and notations.

Let $\mathcal S=(S,\cdot, 1)$ be a monoid and $A$ a nonempty set. If there exists a mapping $-\cdot-: S \times A \rightarrow A$ satisfying the following two conditions: $1\cdot a = a$ and $(s_1\cdot s_2)\cdot a = s_1\cdot (s_2\cdot a)$ then $A$ is said to be a left $S$-act and it is denoted $_SA$.  A mapping $f: {}_{S}A \rightarrow {}_{S}B$ is a homomorphism of $S$-acts (an $S$-homomorphism) provided $f\left( sa\right) = s f\left( a\right)$ holds for all pairs $s\in S, a \in A$.
In compliance with \cite[Example I.6.5.]{KKM} we denote by
$S-\A$ the category of all left $S$-acts with homomorphisms of $S$-acts and $S-\oA$ the category $S-\A$ enriched by an initial object $_S\emptyset$.
Let the monoid $\mathcal S$ contain a (necessarily unique) zero element $0$, which satisfies $0\cdot s = s \cdot 0 = 0$ for all $s\in S$. Then the category of all left $S$-acts $A$ with a unique 
zero element $\theta_A=0A$ and homomorphisms of $S$-acts compatible with zero as morphisms will be denoted $S-\A_0$. Observe that $\theta:=\{0\}$ is the initial object of the category $S-\A_0$ (but not of the category $S-\oA$).

Recall that both of the categories $S-\oA$ and $S-\A_0$ are complete and cocomplete \cite[Remarks II.2.11, Remark II.2.22]{KKM}.
In particular, the coproduct of a system of objects $(A_i, i\in I)$ is
\begin{enumerate}
	\item[(i)] $\coprod_{i\in I} A_i =\dot{\bigcup} A_i $ in  $S$-$\oA$ by \cite[Proposition II.1.8]{KKM} and
	
	\item[(ii)] $\coprod_{i\in I} A_i =\{(a_i)\in\prod_{i\in I} A_i |\ \exists j: a_i=0 \forall i\ne j\}$ in  $S$-$\A_0$ by \cite[Remark II.1.16]{KKM}.
\end{enumerate}

Recall that for a subact $B$ of an act $A$ the \emph{Rees congruence} $\rho_B$ on $A$ is defined by setting $a_1 \rho a_2$ if $a_1 = a_2$ or $a_1, a_2 \in B$ and the corresponding factor act is denoted by $A/B$ (cf. \cite[Definition 4.20]{KKM} )

\section{The functor $\mF: S-\oA \to S-\A_0$}

Throughout the paper, all monoids are considered to {\bf contain the zero element} $0$, in particular, $S$ denotes a monoid $(S,\cdot,1)$ with the zero element $1$ and we suppose that $0\ne 1$.

Let $A$ be a left $S$-act. Since $Sz=0z=z$ for each $z\in 0A=\{0a\mid \in A\}$ we say that $0A$ is a set of zero elements.
Observe that $0A$ can contain more than one element in general
and notice that while a morphism $\alpha: C \to D$ in the category $S-\A_0$ is required to preserve the unique zero, i.~e., $\alpha(\theta_C) = \theta_D$, the category $S-\oA$ is less restrictive: for a morphism $\beta: A \to B$ the image of a zero element of $A$ from the set $0A$ is some zero element of $B$, in other words $\beta(0A) \subseteq 0B$. This leads to the following idea:

Define the functor $\mF$ from the category $S-\A$ to the category $S-\A_0$ as follows: 
\begin{itemize}
	\item for an object $A\in S-\A$, let $\mF (A) = A/0A$, i. e. the $S$-act obtained by gluing all zeroes of $A$ together or, in other words the image of the natural projection onto the Rees factor $\pi_{0A}: A \twoheadrightarrow A / 0A$
	\item for a morphism $\alpha: A\to B$ define $\mF(\alpha)$ in the natural way so that the following square commutes: 
	
	{
	\centering
	
	\begin{tikzcd}
		A \arrow[r, "\pi_{0A}", twoheadrightarrow] \arrow[d, "\alpha"]
		& \mF(A) \arrow[d, "\mF(\alpha)"] \\
		B \arrow[r, "\pi_{0B}", twoheadrightarrow]
		& \mF(B)
	\end{tikzcd}

	}	
\end{itemize}

The morphism $\mF(\alpha)$ can be obtained from the Homomorphism Theorem \cite[Theorem 4.21]{KKM}, since $\ker \pi_{0A} \subseteq \ker \pi_{0B}\alpha$. The explicit formula for $\mF(\alpha)$ is then: 
{	
$\begin{cases}
\mF(\alpha)([a]) = [\alpha(a)] \text{ for } a \not\in 0A\\
\mF(\alpha)(\theta_{\mF(A)}) = \theta_{\mF(B)}
\end{cases}$.
}

Now we formulate the key categorial observation on $\mF$; for the definition of a reflective subcategory we refer, e.g. to \cite[Definition 4.16]{AHS}.

\begin{proposition}
	The category $S-\A_0$ is a reflective subcategory of the category $S-\A$ via the reflector $\mF$.
\end{proposition}
	
	\begin{proof}
		Firstly, we show that $S-\A_0$ is a full subcategory of $S-\A$, i.e. $\Mor_{S-\A}(A,B)=\Mor_{S-\A_0}(A,B)$. For $A,B \in S-\A_0$ consider an $f\in \Mor_{S-\A}(A,B)$. Both $A,B$ being objects of $S-\A_0$ have their respective unique zeros $\theta_A, \theta_B$. Let $f(\theta_A) = b\in B$. Then 
\[
f(\theta_A ) = f(0\theta_A ) = 0f(\theta_A )=\theta_B = 0\theta_B
\]
so $f$ preserves zero and as a consequence $f\in \Mor_{S-\A_0}(A,B)$. The reverse inclusion of morphism sets is clear.
		
		Let now be $A\in S-\A$, $X\in S-\A_0$ and $f: A \to X$ a morphism in $S-\A$. We claim that $\mF(f)$ is the unique morphism in $\Mor_{S-\A_0}(\mF(A), X)$ that makes the following square commute:

		{
			\centering
			
			\begin{tikzcd}
				A \arrow[r, "\pi_{0A}", twoheadrightarrow] \arrow[d, "f", , labels = left]
				& \mF(A) \arrow[d, "\mF(f)"] \\
				X \arrow[r, equal]
				& \mF(X) = X
			\end{tikzcd}
			
		}	 
		
Indeed, if  $\beta :\mF(A)\to X$ satisfies $\beta \pi_{0A} = \mF(f)\pi_{0A}$, then $\beta = \mF(f)$, since $\pi_{0A}$ is surjective. 
\end{proof}

\begin{example}\label{exm_F_is_not} \rm
(1) Let $S$ be an arbitrary non-trivial monoid and consider $A_1=\{\theta\}$ and $A_1=A_1\coprod A_1=\{\theta_1, \theta_2\}$
are two acts in $S-\A$. Then $\mF(A_1)=\mF(A_2)$.

Note that the functor $\mF$ is not faithful since $|\Hom(A_1,A_2)| = 2$, while by applying $\mF$, we get $|\Hom(\mF(A_1), \mF(A_2))| = |\Hom(A_1,A_1)| = 1$.

	(2) The functor $\mF$ is not left-exact (i.e. it does not preserve finite limits): consider the monoid $S = (\Z_2,\cdot,1)$ and the $S$-act $A = \{\theta_A, a\}$ with Cayley graph (omitting unit loops) 
		
		{
			\centering
			
			\begin{tikzcd}
				a \arrow[r, "0"] & \theta_A.
			\end{tikzcd}
			
		}
	Put $B = A \dot{\cup} \theta_S$, an object of $S-\A$ with two zeros. Then $\mF(B\prod B)$ has 6 elements, while $\mF(B)\prod \mF(B) = A \prod A$ is a 4-element act.
\end{example}

The previous examples show that $\mF(A)$ cannot be considered in a reasonable way an analogy of localization or completion of $A$.

\begin{lemma}\label{pres_coprod}
	The functor $\mF$ preserves coproducts.
\end{lemma}

\begin{proof}
	 Since $\mF$ is a reflector, hence a left adjoint (of the embedding functor $S-\A_0 \hookrightarrow S-\A$), it preserves colimits by the dual assertion of \cite[Theorem 1, page 114]{MacL}.
\end{proof}

Recall that an act $P$ is projective, if for any pair of acts $A, B $,
a homomorphism  $\alpha: P \rightarrow B$ and an epimorphism
$\pi:A\rightarrow B$, there exists a morphism $\overline{\alpha}: P \rightarrow A$ in $\mathcal{C}$ such that $\alpha = \pi\overline{\alpha}$.

\begin{lemma}\label{PreserveProj}
	Let $P\in S-\A$ be projective. Then $\mF(P)$ is projective in $S-\A_0$.
\end{lemma}
\begin{proof}
Let the projective situation in $S-\A_0$ be given:

{
	\centering
	
	\begin{tikzcd}[row sep=tiny]
		 \mF(P) \arrow[dr,"f"] \\
	 &	A 		\\
		B \arrow[ur, "\pi", twoheadrightarrow, labels=below  right ]  
	\end{tikzcd}

}

Since $S-\A_0$ is a subcategory of $S-\A$ and we have $\pi_{0P}: P \twoheadrightarrow \mF(P)$, the projectivity of $P$ provides a morphism $\alpha: P \to B$ in $S-\A$ such that $\pi\alpha = f\pi_{0P}$; furthermore, $\ker \pi_{0P} \subseteq \ker \alpha$, hence $\alpha$ factorizes through $ \pi_{0P}$ via some $\alpha^{\prime}: \mF(P) \to B$:

{
	\centering
	
	\begin{tikzcd}[row sep=tiny]
		P \arrow[dd, "\alpha", rightsquigarrow, labels = left] \arrow[r, "\pi_{0P}"] & \mF(P) \arrow[dr,"f"]  \arrow[to = 3-1,"\alpha^{\prime}",rightsquigarrow]\\
		& &	A 		\\
		B \arrow[r, "\pi_{0B} = id", labels=below ] & B \arrow[ur, "\pi", twoheadrightarrow, labels=below  right ]  
	\end{tikzcd}

}

In total: $\pi\alpha^{\prime}\pi_{0P} = \pi\alpha = f\pi_{0P}$ and since $\pi_{0P}$ is an epimorphism, we get $\pi\alpha^{\prime} = f$.
\end{proof}

Let as observe that the description of projectivity in $S-\A_0$  works similarly as in $S-\A$ \cite[Theorem III.17.8]{KKM}.

\begin{lemma}\label{indec-proj} For an indecomposable projective act $A$ in 
$S-\A_0$ there exists an idempotent $e\in S$ such that $A\cong Se$.
\end{lemma}
\begin{proof} 
 We follow the arguments of the proof of \cite[Proposition III.17.7]{KKM}.

By \cite[Lemma 4.4]{DZ21} there exist a retraction $p: S\to A$
and a coretraction $i:A\to S$ such that $pi=\id_A$.
If we put $e=ip(1)$ it is easy to
see that $e=ip(1)=ip(e)=e^2$ and $A\cong i(A) = Se$.
\end{proof}

\begin{proposition}\label{proj} An act $A$ is projective in
$S-\A_0$ if and only if there exist idempotents $e_i$, $i\in I$
such that $A=\coprod_{i\in I} Se_i$.
\end{proposition}

\begin{proof} We follow the arguments of the proof of \cite[Theorem III.17.8]{KKM}.

By \cite[Theorem 4.3]{DZ21}, $A$ is an projective act if and only if it
is isomorphic to a direct sum of indecomposable projective acts.
 Since every indecomposable projective act
is isomorphic to $Se$ for some idempotent $e$ by Lemma~\ref{indec-proj}, it remains to observe that for each act $Se$, where $e$ is an idempotent, 
the inclusion morphism $i: Se\to S$ forms a coretraction and 
the projection $p: S\to Se$ given by the rule $p(s)=se$ forms a retraction and since $S$ is projective, $Se$ is projective, too.
\end{proof}

Now, we show that locally cyclic acts contains only one zero-element.

\begin{lemma}\label{unique_zero}
Any cyclic $S$-act $A$ contains a unique zero element $\theta_A$.
\end{lemma}
\begin{proof} 
Since the act $A$ is cyclic, there exists a $g\in G$ for which $A= Sg$. Let $\theta$ be a zero element of $A$. Then there exists an $s\in S$ such that
$\theta=s\cdot a$, and so 
$\theta=0\theta =0sa=0a$. Thus $0A=\{\theta\}$.
\end{proof}

Recall that an $S$-act is locally cyclic, if for any pair of elements $a_1,a_2 \in A$ there exists a $b \in A$ with $a_i \in Sb$ for $i=1,2$. 

\begin{corollary}\label{im_cyclic}
If $A$ is a  locally cyclic $S$-act, then it contains a unique zero element $\theta_A$, the morphism $\pi_{0A}$ is bijective, and we can assume $\mF(A) = A$.
\end{corollary}

For any act $A\in S-\A$ we can consider the one-element $S$-act $_S\theta$ being adjoined, $A\, \dot{\cup}\, _S\theta \simeq A \coprod _S\theta$.  Therefore define a property $\mathcal{P}$ of an $S$-act $A\in S-\A$ to hold \textit{\termin} in the case $A \simeq A^{\prime}\, \dot{\cup}\, \dot{\bigcup}_{i\in I}\,_S\theta$, $A^{\prime}$ cannot be decomposed as $A^{\prime\prime} \,\dot{\cup}\, _S\theta$ and it has the property $\mathcal{P}$. Call then $A^{\prime}$ the \emph{substantial summand} of $A$. Finally, a subact $B$ of $A$ is said to be \emph{superfluous} if $B \cup C \neq A$ for each proper subact $C$ of $A$.

Note that in $S-\A_0$ the adjunction of $_S\theta$ is trivial, since $A \coprod _S\theta \simeq A$, and let us list now some elementary properties of zero elements and substantial summands.

\begin{lemma}\label{substantial} Let $A\in S-\A$. 
\begin{enumerate}
\item If $\emptyset \ne C\subseteq 0A$, then 
$C=\dot{\bigcup}_{c\in C}\{c\}\cong \coprod_{c\in C}\theta$ is a subact of $A$.
\item If $B$ is a subact of $A$ satisfying $A=B\cup 0A$,
then $A\cong B\coprod(0A\setminus B)\cong B\coprod(\coprod_{c\in 0A\setminus B}\theta)$.
\item $A$ contains a substantial summand.
\item If $A$ is indecomposable, then it is the substantial summand of itself and $0A$ is a superfluous subact of $A$.
\end{enumerate}
\end{lemma}
\begin{proof}
 (1) It is clear as $Sc=c=0c$ for all $c\in 0A$.

(2) Since $A=B\dot{\cup}(0A\setminus B)$ and $(0A\setminus B)\subseteq 0A$, the claim follows from (1).

(3) By \cite[Theorem I.5.10]{KKM} there exists, up to a permutation, 
a unique decomposition $A=\dot{\bigcup}_{i\in I}A_i$ of $A$ into indecomposable subacts. If we put $B=\dot{\bigcup}\{A_i\mid A_i\nsubseteq 0A\}$ and $C=\dot{\bigcup}\{A_i\mid A_i\subseteq 0A\}$,
then $A=B\dot{\cup}C\cong B\coprod(\coprod_{c\in C}\theta)$ by (1) and (2), hence $B$ is the substantial summand of $A$ by the uniqueness of the decomposition.

(4) If $B$ is a subact of an indecomposable act $A$ such that 
$B\cup 0A=A$, then $A\cong B\coprod (\coprod_{\theta\in 0A\setminus B}\theta)$ by (2), hence we have $0A\subseteq B$ and $B=A$.
\end{proof}

\begin{lemma}\label{preim_cyclic}
	If $A\in S-\oA$ such that $\mF(A)$ is a nonzero cyclic $S$-act, then $A$ is, \termin, cyclic.
\end{lemma}
\begin{proof}
As $\mF(A)=\pi_{0A}(A)$ is cyclic, there exists $g\in A$ such that 
$\mF(A)=S\pi_{0A}(g)$, hence $A=Sg\cup 0A$ by the definition of the Rees factor. Since $A\cong Sg\coprod(\coprod_{c\in 0A\setminus Sg}\theta)$ by Lemma~\ref{substantial}(2), $A$ is, \termin, cyclic.
\end{proof}


Note that the image nor the preimage under $\mF$ of an indecomposable act may not be indecomposable, as the following examples illustrate:

\begin{example}\rm
	(1) Consider the monoid $S$ from Example~\ref{exm_F_is_not}(2)
	 and the $S$-act $A = \{\theta_A, a,b\}$ with Cayley graph (omitting unit loops)
	
	{	\centering
		
		\begin{tikzcd}
			a \arrow[r, "0"] & \theta{_A} & \arrow["0"', l]  b. 
		\end{tikzcd}
		
	}
	Then $A$ is indecomposable in $S-\oA$, but $\mF(A) = A$ is decomposable in $S-\A_0$.
	
	(2) For any indecomposable $A \in S-\A_0$ and a nonempty index set $I$, the act $B=A \, \dot{\bigcup}_{i\in I} \, ({\theta_i}) \in S-\oA$ is decomposable with $\mF(B) \simeq \mF(A)$ indecomposable.

\end{example}

Recall that projective objects of both categories $S-\oA$ and $S-\A_0$ are isomorphic to coproducts (in the respective category) $\coprod_{i\in I} Se_i$ of cyclic $S$-acts of the form $Se_i$ with $e_i\in S$ idempotents by \cite[Proposition 17.8]{KKM} and Proposition~\ref{proj}.

\begin{example} \rm
	The functor $\mF$ is not bijective on the class of projective objects of $S-\A$ for any monoid $S$, as there exists a non-projective $A\in S-\A$ with $\mF(A)$ projective: consider the coproduct $A = S_1 \dot{\cup} S_2$, where $S_i\cong S$. Then $\mF(S_1 \dot{\cup} S_2)$ is not projective in  $S-\A$  while
 $\mF(\mF(S_1 \dot{\cup} S_2))= \mF(S_1 \dot{\cup} S_2)$ is projective in $S-\A_0$ by Lemma~\ref{pres_coprod}. 
 In particular, $A=\{(a,b)\in\Z^2\mid a=0\lor b=0\}$ is not projective in $\Z-\A$ and $\mF(A)\cong A$ is projective in $\Z-\A_0$.	
\end{example}

\section{Perfect monoids}

Recall that for an act $A$, a pair $(C,f)$ is a \emph{cover} provided $f:C\to A$ is an epimorphism, and for any proper subact $C^{\prime} \subset P$ the restriction 
$f|_{C^{\prime}}: C^{\prime}\to A$ is not an epimorphism in the corresponding category. 
A cover $(P,f)$ is called \emph{projective} in case $P$ is 
projective (cf. \cite[chapter 17]{KKM}). Note that a projective cover is maximal among all covers.

\begin{lemma}\label{cover}
	Let $(P, f)$ be a projective cover of $A$ in the category $S-\A$. Then $(\mF(P), \mF(f))$ is a projective cover of $\mF(A)$ in the category $S-\A_0$.
\end{lemma}

\begin{proof}
By Lemma~\ref{PreserveProj}, $\mF(P)$ is projective. Let $Q\subsetneq \mF(P)$ be a subact and put $\tilde{Q}=\pi_{0P}^{-1}(Q)$.
Then $0P\subseteq \tilde{Q}\subsetneq P$, hence $f(\tilde{Q})\ne A$ by the hypothesis and $0A=0f(P)=f(0P)\subseteq f(\tilde{Q})$, as $f$ is surjective. It implies that $\pi_{0A}(\tilde{Q})\ne \pi_{0P}(A)=\mF(A)$, thus
\[
\mF(f)(Q)=\mF(f)(\pi_{0P}(\tilde{Q}))= \pi_{0A}f(\tilde{Q}))\ne \mF(A).
\]
\end{proof}

In analogy with module categories, call a monoid \emph{left perfect} (\emph{left 0-perfect}) if each $A\in S-\oA$ ($A\in S-\A_0$) has a projective cover (cf. \cite{DZ21, I, K}). Let us recall a characterization of left-perfect monoids:

\begin{theorem}\cite[1.1]{I}\label{char_perf}
	A monoid $S$ is left-perfect if and only if each cyclic $S$-act has a projective cover and every locally cyclic $S$-act is cyclic.
\end{theorem}

\begin{proposition}\label{0-perfect is perfect}
	If a monoid $S$ is left-0-perfect, then it is left perfect.
\end{proposition}
\begin{proof}
	Suppose that $S$ is left-0-perfect and let us prove the two conditions from Theorem~\ref{char_perf}.
	
First suppose that $A\in S-\oA$ is a locally cyclic act. Then
$A$ contains a unique zero $\theta_A$ and $A\cong\pi_{0A}(A)$ can be
considered an act of the category $S-\A_0$ by Corollary~\ref{im_cyclic}. 
Let $f:P\to A$
be a projective cover in $S-\A_0$. We show that $P$ is indecomposable.

Applying Proposition~\ref{proj} assume to the contrary that $P=P_1\coprod P_2$ is a non-trivial decomposition, where $P_1=Se$ is cyclic.
Since $f(P)=A=f(P_1)\cup f(P_2)$, there exists $y\in A\setminus f(P_1)$ and  there exists $z\in f(P_2)$ such that $f(e), y\in Sz \subseteq f(P_2)$. Hence $f(P_1)\subseteq Sz\subseteq f(P_2)$, and
so $A=f(P_1)\cup f(P_2)= f(P_2)$, a contradiction.

Since $P$ is indecomposable, $P\cong Se$ for an idempotent $e\in S$ by Lemma~\ref{indec-proj}, which implies that $A$ is cyclic. Furthermore, as $Se$ is a projective act also in the category $S-\oA$ by \cite[Proposition 17.8]{KKM}, the morphism $f$ constitutes a projective cover in $S-\oA$.
\end{proof}

\begin{theorem}\label{char_0perfect}
A monoid is left perfect if and only if  it is left 0-perfect.
\end{theorem}
\begin{proof}
The direct implication follows from Lemma~\ref{cover} and the reverse one is proven by Proposition~\ref{0-perfect is perfect}.
\end{proof}

\begin{example} \rm
By \cite{I}, the examples of monoids which are left perfect (the argument does not require the zero) comprise: monoid of square matrices over a division ring, and finite monoids. By Theorem~\ref{0-perfect is perfect}, the former is also left-0-perfect, while the latter in case it contains a zero element.
	
On the other hand, in the case of another class of perfect monoids (without zero) mentioned in \cite{I}, groups, the presented result cannot be employed, as adding 0 to a group may in general change the situation notably (see Example~\ref{exm_Q} below).
\end{example}

\section{Steady monoids}

An $S$-act $A$ is called \emph{hollow} if each of its proper subacts is superfluous (cf. \cite[Definition 3.1]{Kho-Ro}). 
It is easy to see that hollow acts are indecomposable in both categories $S-\oA$ and $S-\A_0$ (see \cite[Theorem 3.4]{Kho-Ro} and \cite[Propositions 5.6 and 6.6]{DZ21}).

In compliance with \cite{DZ21} call an act $C \in S-\oA$ ($\in S-\A_0$, resp.) \emph{compact}, if the corresponding covariant $\Hom$-functor commutes with coproducts, i.e. for any family $(A_i, i \in I)$ of $S$-acts in the given category, for the natural functor $\Hom(C, - ): S-\oA \to \Set$ ($S-\A_0 \to \Set$, resp.) we have a surjective natural morphism
\[ \Hom(C, \coprod_{i\in I} A_i) \to \coprod_{i\in I} \Hom(C, A_i)\to 0.\]

Recall that an act in the category $S-\oA$ is compact if and only if it is hollow by \cite[Proposition 6.6]{DZ21}.
It is easy to see that cyclic acts are compact and
we say that a monoid $S$  is {\it left steady} (resp. 
{\it left $0$-steady}) if every compact act in the category $S-\oA$
(resp. $S-\A_0$) is cyclic (see \cite[6.2]{DZ21}).

\begin{lemma}\label{hollow}
Let $A$ be an act in $S-\A$ such that $0A$ is superfluous in $A$.
Then $A$ is hollow in the category  $S-\A$  if and only if $\mF(A)$ is hollow in the category $S-\A_0$.
\end{lemma}

\begin{proof} Let $A$ be hollow and $\mF(A)=B_1\cup B_2$ for subacts $B_i$, $i=1,2$. Then $A=\pi_{0A}^{-1}(B_1)\cup \pi_{0A}^{-1}(B_2)$, hence 
there exists $i$ such that $A=\pi_{0A}^{-1}(B_i)$ and so 
$\mF(A)=\pi_{0A}(A)=B_i$. Thus $\mF(A)$ is hollow.

Conversely, suppose that $A=B_1\cup B_2$ for subacts $B_i$ of $A$ and $i=1,2$. Then $\mF(A)=\pi_{0A}(B_1)\cup \pi_{0A}(B_2)$ and so 
there exists $i$ for which  $\mF(A)=\pi_{0A}(B_i)$. It implies that
$B_i\cup 0A=A$, thus $B_i=A$ since $0A$ is superfluous in $A$.
\end{proof}

Recall a description of the monoid structure via a property of hollow acts, which is employed in the next result:

\begin{lemma}\cite[Lemma 3.8]{Kho-Ro}\label{Kho-Ro}
A monoid $S$ satisfies the ascending chain condition on cyclic subacts of an arbitrary $S$-act if and only if every hollow act in $S-\A$  is cyclic.
\end{lemma}

\begin{theorem}\label{char_0steady}
A monoid $S$ is left 0-steady  if and only if it satisfies the ascending chain condition on cyclic subacts.
\end{theorem}

\begin{proof} By Lemma~\ref{Kho-Ro} it is enough to prove that 
$S$ is left 0-steady if and only if every hollow $S$-act in $S-\A$ is cyclic.

Let $S$ be left 0-steady and let $A$ be a hollow $S$-act in $S-\A$.
Since $A$ is indecomposable, $0A$ is superfluous by Lemma~\ref{substantial}. Applying Lemma~\ref{hollow} we obtain that $\mF(A)$ is hollow in the category $S-\A_0$, which implies that $\mF(A)$ is compact in $S-\A_0$  by \cite[Proposition 6.6]{DZ21}. 
Thus $\mF(A)=\pi_{0A}(A)$ is cyclic by the hypothesis and by Lemma~\ref{preim_cyclic} we get $A=Sa\cup 0A$.
Finally, since $0A$ is superfluous, $A$ is cyclic. 

Conversely, suppose that $A$ is a compact act in the category $S-\A_0$.
Then it is hollow by \cite[Proposition 6.6]{DZ21}, and  so indecomposable. Now, it follows from Lemmas~\ref{substantial} and \ref{hollow}
 that $A\cong \mF(A)$ is hollow in $S-\A$, hence it is cyclic by the hypothesis.
\end{proof}

We conclude the paper by an example.
 
\begin{example}\label{exm_Q} \rm
Any group $G$ is right steady by \cite[Example 6.7(1)]{DZ21}, however 0-steadiness of a monoid $G_0$ obtained from $G$ by adding a zero element depends on the structure of subgroups by the last theorem. In particular $\Q^*$ is steady, while $\Q$ is not 0-steady.
\end{example}

\end{document}